\documentclass[11pt,twoside]{article}
\usepackage[mathscr]{euscript}
\usepackage{secdot}
\usepackage[utf8]{inputenc}
\usepackage{amsmath, amsthm, amscd, amsfonts, amssymb, color}
\usepackage[bookmarksnumbered,plainpages]{hyperref}

\date{}

\begin{document}

\centerline{}

\centerline {\Large{\bf A Few Properties of $\delta$-Continuity and $\delta$-Closure on }}
 \centerline {\Large{\bf Delta Weak Topological Spaces}}
\centerline{}
\centerline{\textbf{Sanjay Roy}}
\centerline{Department of Mathematics, Uluberia College}
\centerline{Uluberia, Howrah- 711315,  West Bengal, India}
\centerline{e-mail: sanjaypuremath@gmail.com}
\centerline{}

\newcommand{\mvec}[1]{\mbox{\bfseries\itshape #1}}
\centerline{}
\newtheorem{Theorem}{\quad Theorem}[section]

\newtheorem{definition}[Theorem]{\quad Definition}

\newtheorem{theorem}[Theorem]{\quad Theorem}

\newtheorem{remark}[Theorem]{\quad Remark}

\newtheorem{corollary}[Theorem]{\quad Corollary}

\newtheorem{note}[Theorem]{\quad Note}

\newtheorem{lemma}[Theorem]{\quad Lemma}

\newtheorem{example}[Theorem]{\quad Example}
\newtheorem{notation}[Theorem]{\quad Notation}

\newtheorem{result}[Theorem]{\quad Result}
\newtheorem{conclusion}[Theorem]{\quad Conclusion}

\newtheorem{proposition}[Theorem]{\quad Proposition}
\newtheorem{prop}[Theorem]{\quad Property}

\begin{abstract}
The main aim of this paper is to define a weakest topology $\sigma$ on a linear topological space $(E, \tau)$ such that each $\delta$-continuous functional on $(E, \tau)$ is $\delta$-continuous functional on $(E, \sigma)$ and to find out the relation between the set of these $\delta$-continuous functionals on $(E, \tau)$ and  the set of all $\delta$-continuous functionals on $(E, \sigma)$. Also we find out the closure of a subset $A$ of $E$ on that weakest topology by the $\delta$-closure of $A$ with respect to the given topology.  
\end{abstract}

\section{\textbf{Introduction}}
In functional analysis, weak topology on a linear topological space plays an important role to characterize the set of all continuous functionals. As we are all familiar with weak topology, we recall only a few notions of weak topology which will be needed for the discussion of this paper. A linear topology on a linear space $E$ over $K (=\mathbb{R}$ or $\mathbb{C})$ is a topology that makes each operation of $E$ continuous. A locally convex space is a linear topological space that produces a local base containing convex neighbourhoods of null vector $\theta$. To define a weak topology on a linear topological space, we use a well known theorem: Let $E$ be a linear space over $K (=\mathbb{R}$ or $\mathbb{C})$ and $\{P_\alpha:\,\alpha\in \Lambda\}$ be the collection of some seminorms  on $E$. Let $\mathcal{V}=\lbrace B_{P_\alpha}(\theta, r_\alpha):\, r_\alpha>0, \, \alpha\in \Lambda\rbrace$, where $B_{P_\alpha}(\theta, r_\alpha)=\{x\in E: P_\alpha(x)<r_\alpha\}$ $\forall\, \alpha\in \Lambda$ and $\forall \,r_\alpha> 0$. Then the collection $\mathcal{U}$ of all finite intersection of the members of $\mathcal{V}$ forms an absolutely convex local base for the weakest locally convex topology on $E$ such that each $p_\alpha$ is continuous for $\alpha\in \Lambda$. If $\Gamma$ is the collection of some linear functional on a linear space $E$ over $K\,(=\mathbb{R}$ or $\mathbb{C})$ then $P_f(x)=|f(x)|$  is a seminorm for all $f\in \Gamma$. So by these collection of seminorms, we can always construct a weakest locally convex topology on $E$ such that each $P_f$ is continuous. This topology is called the $\Gamma$-topology on $E$ and is denoted by $\sigma (E, \Gamma)$. Now in particular, if $E$ is a linear topological space and $\Gamma=E^{*},$ the collection of all continuous linear functionals on $E$, the $\Gamma$-topology is called the weak topology $\sigma(E, E^{*})$ on $E$. In a linear space $E$, it can be shown that $(E, \sigma(E, \Gamma))^{*}=\Gamma$ if $\Gamma$ is a subspace of $E^{\prime}$, the linear space of all linear functionals. So, in particular $(E, \sigma(E, E^{*}))^{*}=E^{*}$.

Now our target is to define another type of $\Gamma$-topology namely, delta weak topology. To define this type of $\Gamma$-topology, we first introduce a few notions which are already established. 

In 1979, T. Noiri \cite{Noiri} introduced the concept of $\delta$-cluster point. Then he defined $\delta$-closure of a set by the collection of all $\delta$-cluster points. Thereafter he introduced the definition of $\delta$-closed set with the help of $\delta$-closure and defined a $\delta$-open set as a complement of $\delta$-closed set. There he had given the notion of $\delta$-continuous function  and established its few properties. According to Noiri, a function $f:X\rightarrow Y$ is said to be $\delta$-continuous if for each $x\in X$ and each open neighbourhood $V$ of $f(x)$, there exists an open neighbourhood $U$ of $x$ such that $f(int[cl(U)])\subseteq int[cl(V)]$.  
In 1993, S. Raychaudhuri and M. N. Mukherjee \cite{Raychaudhuri} generalized the concepts of $\delta$-continuity and $\delta$-open set and called these general definitions  $\delta$-almost continuity and $\delta$-preopen set respectively.  In 1993, N. Palaniappan and K. C. Rao \cite{Palaniappan} introduced a concepts of regular generalized closed set and regular generalized open set. Thereafter many works have been done on another concepts of $\delta$-semiopen set \cite{Park, Caldas 1}, $\delta$-set \cite{Saleh}, Semi$^{*}\delta$-open set \cite{Missier 2}, regular$^{*}$ open set \cite{Missier 1, Missier 3}. In 2005, M. Caldas \cite{Caldas 2} has investigated to find out some applications of $\delta$-preopen sets.

In 2021, R. M. Latif \cite{Latif} has defined a $\delta$-continuous function in an another way. In this paper we have taken the definition of $\delta$-continuous function as stated by R.M. Latif and by these collection of $\delta$-continuous functions, we have introduced the concept of delta weak topology and tried to find out the behaviour of the collection of all $\delta$-continuous functionals and of the $\delta$-closure of a set on a delta weak topological space. 

\section{Preliminaries}
Here we have given a few basic definitions and theorems which will be needed in the sequel. Throughout this paper we have taken the field $K$ as either the field of real numbers $\mathbb{R}$ or the field of complex numbers $\mathbb{C}$.

\begin{definition}\cite{Noiri}
An open set $U$ of a topological space $(X,\tau)$ is said to be a Regular Open Set if $U=$ $int[cl(U)]$.
\end{definition}

\begin{definition}\cite{Noiri}
A subset $V$ of a topological space $(X,\tau)$ is said to be regular closed if $V=$ $cl[int(V)]$.
\end{definition}

\begin{theorem}\cite{Roy}
$V$ is a regular closed set iff $X\setminus V$ is a regular open set.
\end{theorem}

\begin{result}\cite{Roy}
Let $(X, \tau)$ be a topological space and $U\in \tau$. Then $int[cl(U)]$ is regular open.
\end{result}

\begin{note}\cite{Roy}
 A subset $U$ of $(X,\tau)$ is called $\delta$-open set if for each $x\in U$, $\exists$ a regular open set $P$ such that $x\in P$ $\subseteq U$.
\end{note}

\begin{theorem}\cite{Roy}
$V$ is $\delta$-closed iff $V$ is the intersection of all regular closed set containing $V$.
\end{theorem}

\begin{theorem}\cite{Roy}
The intersection of finitely many regular open sets in any topological space is regular open.
\end{theorem}

\begin{theorem}\cite{Roy}
The arbitrary union of $ \delta$-open sets in any topological spaces is $ \delta$-open.
\end{theorem}

\begin{theorem}\cite{Roy}
The intersection of finitely many $\delta$-open sets in any topological spaces is $\delta$-open.
\end{theorem}

\begin{theorem}\cite{Roy}
The intersection of finite collection of $\delta$-closed sets is $\delta$-closed.
\end{theorem}

\begin{definition}\cite{Latif}
A mapping $f:(X,\tau)\rightarrow (Y,\sigma)$ is said to be $\delta$-continuous if the pre-image of every open subset of $Y$ is a $\delta$-open subset of $X$.
\end{definition}

\begin{definition}
Let $M$ be an absorbing subset of a linear space $E$ over $K$. Let $C_M(x)=\{\lambda>0: \frac{x}{\lambda}\in M\}$. Clearly $C_M(x)\neq \phi$. The mapping $p_M: E\rightarrow \mathbb{R}^+\cup\{0\}$ defined by $p_M(x)=\inf C_M(x)$ for all $x\in E$ is called the Minkowski functional of $M$.
\end{definition}

\begin{definition}
A real valued function $p$ defined on a linear space $E$ over $K$ is called a seminorm if the following conditions are satisfied :\\
$(i)$ $p(x+y)\leq p(x)+p(y)$ for all $x, y\in E$,\\
 $(ii)$ $p(\alpha x)=|\alpha| p(x)$ for all $\alpha\in K$ and for all $x\in E$.
\end{definition}

\begin{theorem}
 Let $M$ be an absorbing subset of a linear space $E$ over $K$ and $p_M$ be the Minkowski functional of $M$. Then \\
 $(i)$ $p_M(t x)=t \,p_M(x)$ for all $x\in E$ and for all $t> 0$,\\
 $(ii)$ $M$ is convex $\Rightarrow$ $p_M(x+y)\leq p_M(x)+p_M(y)$ for all $x, y\in E$,\\
 $(iii)$ $M$ is absolutely convex $\Rightarrow$ $M$ is a seminorm.
 \end{theorem}

\begin{theorem}\label{th5}
In a locally convex space $E$ over $K$, the absolutely convex closed neighbourhoods of $\theta$ forms a local base at $\theta$.
\end{theorem}

\begin{theorem}
Let $E$ be a linear topological space and $A\subseteq E$. Then\\
$(i)$ $A$ is balanced and $\theta\in int A\Rightarrow int A$ is balanced.\\
$(ii)$ $A$ is convex $\Rightarrow int A$ is convex.
\end{theorem} 

\begin{theorem}
Let $E$ be a complex vector space. If $f$ is a complex linear functional, its real part is a real linear functional. Also, every real linear functional $f_0: E\rightarrow \mathbb{R}$ is the real part of a unique complex linear functional $f: E\rightarrow \mathbb{C}$ where $f(x)=f_0(x)-if_0(ix)$ for all $x\in E$.
\end{theorem}

\begin{theorem}$($ Hahn-Banach theorem $)$\\
Suppose $(i)$ $M$ is a subspace of a real linear space $E$,\\
$(ii)$ $p:E\rightarrow \mathbb{R}$ is a function satisfying $p(x+y)\leq p(x)+p(y)$ and $p(tx)=t p(x)$ for all $x, y\in E$ and $t\geq 0$,\\
$(iii)$ $f_0$ is a linear functional on $M$ such that $f_0(x)\leq p(x)$ for all $x\in M$.\\
Then there exists a real linear functional $f$ on $E$ such that $f(x)=f_0(x)$ for all $x\in M$ and $-p(-x)\leq f(x)\leq p(x)$ for all $x\in E$.
\end{theorem}

\section{Delta Weak Topology}
\begin{definition}
Let $(E, \tau)$ be a linear topological space over the field $K= (\mathbb{R}$ or $\mathbb{C})$. The collection of all linear $\delta$-continuous functionals is called $\delta$- topological dual of $E$ and it is denoted by $\delta$-$E^{*}$. As every $\delta$-continuous mapping $f$ is   continuous, $\delta$-$E^{*}\subseteq E^{*}$.
\end{definition}

\begin{definition}\label{d2}
Let $(E, \tau)$ be a linear topological space over the field $K= (\mathbb{R}$ or $\mathbb{C})$. Let $\delta$-$E^{*}$ be the $\delta$- topological dual of $E$. For each element $f\in \delta$-$E^{*}$ we  define a seminorm $p_f(x)=|f(x)|$ where $x\in E$. Then the collection of all such seminorms $\{p_f:\, f\in \delta$-$E^{*}\}$ induces a topology on $E$. This topology is called the delta weak topology $(\delta$- weak topology $)$ and it is denoted by $\sigma (E, \delta$-$E^{*})$.
\end{definition}

\begin{remark}
From the above Definition \ref{d2}, it is clear that $\sigma (E, \delta$-$E^{*})$ is the smallest locally convex topology on $E$ defined by the seminorms $\{p_f:\, f\in \delta$-$E^{*}\}$ such that each $p_f$ is continuous on $(E, \sigma (E, \delta$-$E^{*}))$ where $p_f(x)=|f(x)|$ for all $x\in E$. A local base at $\theta$, the null vector of $E$ on $(E, \sigma (E, \delta$-$E^{*}))$ consists the sets of the form $v(f_1, f_2,\cdots, f_n; r_1, r_2,\cdots,r_n)=\cap_{i=1}^{n}\{x\in E: \,|f_i(x)|<r_i\} $ where $n\in \mathbb{N}$, $f_i\in \delta$-$E^{*}$ and $r_i> 0$ for $i=1, 2, \cdots, n$.
\end{remark}

\begin{theorem}\label{th1}
If $(E, \tau)$ is a linear topological space then  $\sigma (E, \delta$-$E^{*})\subseteq\sigma (E, E^{*})$, where $\sigma (E, E^{*})$ is the weak topology on $E$.
\end{theorem}

\begin{proof}
Since $f\in\delta$-$E^{*}$ implies that $f\in E^*$, the subbasic open neighbourhood $B_f^\delta(\theta, r)=\{x\in E: |f(x)|< r\}$ of $\theta$ in $\sigma (E, \delta$-$E^{*})$ is a subbasic open neighbourhood of $\theta$ in $\sigma (E, E^{*})$. So,  $\sigma (E, \delta$-$E^{*})\subseteq\sigma (E, E^{*})$.
\end{proof}

\begin{note}\label{n2}
If $(E, \tau)$ is a linear topological space then we know that $\sigma (E, E^{*})\subseteq \tau$. So from the above Theorem \ref{th1}, it follows that $\sigma (E, \delta$-$E^{*})\subseteq\sigma (E, E^{*})\subseteq \tau$.
\end{note}

\begin{note} \label{n1}
Let $(E, \tau)$ be a linear topological space. Then $(E, \sigma (E, \delta$-$E^{*}))^*=\delta$-$E^{*}$ as $\delta$-$E^{*}$ is a linear subspace of $E^{'}$, the linear space of all linear functionals, that is, the collection of all continuous linear functionals of $(E, \sigma (E, \delta$-$E^{*}))$ is the set $\delta$-$E^{*}$.
\end{note}

\begin{theorem}\label{th2}
Let $(E, \tau)$ be a linear topological space over $K$. Then $\sigma (E, \delta$-$E^{*})$ is the weakest topology on $E$ such that each member of $\delta$-$E^{*}$ is $\delta$-continuous. 
\end{theorem}

\begin{proof}
Let $f\in \delta$-$E^{*}$.   We first show that $f$ is $\delta$-continuous on $(E, \sigma (E, \delta$-$E^{*}))$. As $(E, \sigma (E, \delta$-$E^{*}))$ is a Locally convex space, it is enough to show that for any open neighbourhood $U$ of $0$ in $K$, there exists $V\in \sigma (E, \delta$-$E^{*})$ such that  $int_\sigma cl_\sigma V=V$ and $\theta\in V\subseteq f^{-1}(U)$, where   $int_\sigma$ and $cl_\sigma$ are denoted respectively by interior and closure operators of $(E, \sigma (E, \delta$-$E^{*}))$. 

Since $f\in \delta$-$E^{*}$, there exists $W\in \tau$ such that  $int_\tau cl_\tau W=W$ and $\theta\in W\subseteq f^{-1}(U)$ where   $int_\tau$ and $cl_\tau$ are denoted respectively by interior and closure operators of $(E, \tau)$. Then there exists a balanced neighbourhood $B$ of $\theta$ on $(E, \tau)$ such that  $\theta\in B\subseteq W\subseteq f^{-1}(U)$. Since $f$ is linear on a linear topological space, $f$ is an open mapping on $(E, \tau)$. So $f(B)$ is a balanced neighbourhood of $0$ in $U$.  So there exists $r>0$ such that $B\subseteq f^{-1}(-r, r)=\{x\in E: |f(x)|<r\}\subseteq W\subseteq f^{-1}(U)$ and $(-r, r)\subseteq U$. Let $V=\{x\in E: |f(x)|<r\}$. Then $V$ is a subbasic open neighbourhood of $\theta$ on $(E, \sigma (E, \delta$-$E^{*}))$ and  $\theta\in V\subseteq f^{-1}(U)$. Now obviously $V\subseteq int_\sigma cl_\sigma V$. \\
Now since $p_f$ is continuous on $(E, \sigma (E, \delta$-$E^{*}))$, $\{x\in E: |f(x)|=p_f(x)\leq r\}$ is closed in $\sigma (E, \delta$-$E^{*})$. So, $cl_\sigma V\subseteq \{x\in E: |f(x)|\leq r\}$. Then 
$int_\sigma cl_\sigma V\subseteq int_\sigma\{x\in E: |f(x)|\leq r\}\subseteq int_\tau \{x\in E: |f(x)|\leq r\}= \{x\in E: |f(x)|< r\}=V$ as $f$ is continuous on $(E,\tau)$.
Thus $V$ is a regular open set on $(E, \sigma (E, \delta$-$E^{*}))$. So each member of $\delta$-$E^{*}$ is $\delta$-continuous on $(E, \sigma (E, \delta$-$E^{*}))$. \\
Now let $\gamma$ be any topology on $E$ such that each member of $\delta$-$E^{*}$ is $\delta$-continuous on $(E, \gamma)$. Let  $\{x\in E: |f(x)|< r\}$ be a subbasic $\sigma (E, \delta$-$E^{*})$-open neighbourhood of $\theta$, where $f\in \delta$-$E^{*}$. Since   $f\in \delta$-$E^{*}$, $f$ is $\delta$-continuous on $(E, \gamma)$, that is, continuous on $(E, \gamma)$.  Thus $\{x\in E: |f(x)|< r\}\in \gamma$. This completes the proof.
\end{proof}

\begin{theorem}\label{th3}
Let $(E, \tau)$ be a linear topological space over $K$. Then $\delta$-$(E, \sigma (E, \delta$-$E^{*}))^{*}= \delta$-$E^{*}$.
\end{theorem}

\begin{proof}
From the above Theorem \ref{th2}, we can say that $\delta$-$E^{*}\subseteq\delta$-$(E, \sigma (E, \delta$-$E^{*}))^{*}$. 
Let $f\in \delta$-$(E, \sigma (E, \delta$-$E^{*}))^{*}$ and $U$ be any open neighbourhood of $f(\theta)$. Then there exists $V\in \sigma (E, \delta$-$E^{*})$ such that $\theta\in V\subseteq f^{-1}(U)$ and $int_\sigma cl_\sigma V=V$. Since $V\in \sigma (E, \delta$-$E^{*})$, there exist $f_1, f_2, \cdots, f_n \in \delta$-$E^{*}$ such that $V=\cap_{i=1}^n\{x\in E: |f_i(x)|<r_i\}$, where $r_i> 0$ for $i=1, 2, \cdots, n$. Since each $f_i\in \delta$-$E^{*}$, $\{x\in E: |f_i(x)|<r_i\}= f_i^{-1}(B(0, r_i))$ is a $\delta$-open on $(E, \tau)$ for $i=1, 2, \cdots, n$. So there exists regular open set $V_i$ on $(E, \tau)$ such that $\theta\subseteq V_i\subseteq \{x\in E: |f_i(x)|<r_i\}$ for $i=1, 2, \cdots, n$.   So $\theta\in \cap_{i=1}^n V_i\subseteq V\subseteq f^{-1}(U)$.\\
Obviously, $\cap_{i=1}^n V_i\subseteq int_\tau cl_\tau[ \cap_{i=1}^n V_i]$. Now $int_\tau cl_\tau [\cap_{i=1}^n V_i]\subseteq int_\tau [  \cap_{i=1}^ncl_\tau V_i]\subseteq  \cap_{i=1}^n[ int_\tau cl_\tau V_i]= \cap_{i=1}^n V_i$ as each $V_i$ is regular on $(E, \tau)$.  So for any open neighbourhood $U$  of $f(\theta)$, there exists a regular open set $\cap_{i=1}^n V_i$ on $(E,\tau)$ such that $\theta\in \cap_{i=1}^n V_i\subseteq f^{-1}(U)$. Again since $E$ is a linear topological space and $f$ is linear, $f\in \delta$-$E^{*}$. Thus $\delta$-$(E, \sigma (E, \delta$-$E^{*}))^{*}\subseteq \delta$-$E^{*}$.  Hence $\delta$-$(E, \sigma (E, \delta$-$E^{*}))^{*}= \delta$-$E^{*}$.
\end{proof}

\begin{theorem}
Let $(E, \tau)$ be a linear topological space. Then $\delta$-$(E, \sigma (E, \delta$-$E^{*}))^{*}= \delta$-$(E, \sigma(E, E^{*}))^{*}$.
\end{theorem}

\begin{proof}
From Theorem \ref{n2}, we have $\sigma (E, \delta$-$E^{*})\subseteq\sigma (E, E^{*})\subseteq\tau$. Let $f\in \delta$-$(E, \sigma (E, E^{*}))^{*}$ and $U$ be any open neighbourhood of $f(\theta)$. Then there exists a regular open set $V$ in $\sigma (E, E^{*})$ such that $\theta\in V\subseteq f^{-1}(U)$. Since $V\in\sigma (E, E^{*})$, there exist continuous functionals $f_1, f_2, \cdots, f_n$ on $(E, \tau)$ such that $V=\cap_{i=1}^n\{x\in E: |f_i(x)|<r_i\}$ where $r_i>0$ for $i=1, 2, \cdots, n$. Obviously $V\in\tau$. So, $V\subseteq int_\tau cl_\tau V$. Again $int_\tau cl_\tau V=int_\tau cl_\tau [\cap_{i=1}^n\{x\in E: |f_i(x)|<r_i\}]\subseteq int_\tau  [\cap_{i=1}^n cl_\tau\{x\in E: |f_i(x)|<r_i\}]\subseteq   \cap_{i=1}^n int_\tau [cl_\tau\{x\in E: |f_i(x)|<r_i\}]=\cap_{i=1}^n \{x\in E: |f_i(x)|<r_i\}=V$ as each $f_i$ is continuous on $(E, \tau)$. So, there exists a regular open set $V$ on $(E, \tau)$ such that $\theta\in V\subseteq f^{-1}(U)$. Again since $E$ is a linear topological space and $f$ is linear, $f\in \delta$-$E^{*}$. Again by Theorem \ref{th3} we have  $\delta$-$(E, \sigma (E, \delta$-$E^{*}))^{*}= \delta$-$E^{*}$. So, $f\in \delta$-$(E, \sigma (E, \delta$-$E^{*}))^{*}$. Thus $\delta$-$(E, \sigma (E, E^{*}))^{*}\subseteq \delta$-$(E, \sigma (E, \delta$-$E^{*}))^{*}$.\\
Let $f\in\delta$-$(E, \sigma (E, \delta$-$E^{*}))^{*}$ and $U$ be any open neighbourhood of $f(\theta)$. Then there exists a regular open set $V$ in $\sigma (E, \delta$-$E^{*})$ such that $\theta\in V\subseteq f^{-1}(U)$. Since $V\in\sigma (E, \delta$-$E^{*})$, $V\in\sigma(E, E^{*})$ as 
$\sigma (E, \delta$-$E^{*})\subseteq \sigma(E, E^{*})$ by Theorem \ref{th1}. So, there exist continuous functionals $f_1, f_2, \cdots, f_n$ on $(E, \tau)$ such that $V=\cap_{i=1}^n\{x\in E: |f_i(x)|<r_i\}$ where $r_i>0$ for $i=1, 2, \cdots, n$. Since $(E, \sigma (E, E^{*}))^{*}= E^{*}$, each $f_i\in (E, \sigma (E, E^{*}))^{*}$. So, $\{x\in E: |f_i(x)|<r_i\}$ is regular open set on $(E, \sigma (E, E^{*}))$, that is, $int^\sigma cl^\sigma\{x\in E: |f_i(x)|<r_i\}=\{x\in E: |f_i(x)|<r_i\}$ for all $i=1, 2, \cdots, n$  where   $int^\sigma$ and $cl^\sigma$ are denoted respectively by interior and closure operators on $(E, \sigma (E, E^{*}))$. Since $V\in\sigma (E, E^{*})$, $V\subseteq int^\sigma cl^\sigma V$. Now  $int^\sigma cl^\sigma V= int^\sigma cl^\sigma [\cap_{i=1}^n\{x\in E: |f_i(x)|<r_i\}]\subseteq int^\sigma  [\cap_{i=1}^ncl^\sigma\{x\in E: |f_i(x)|<r_i\}]\subseteq  \cap_{i=1}^n int^\sigma cl^\sigma\{x\in E: |f_i(x)|<r_i\}=\cap_{i=1}^n \{x\in E: |f_i(x)|<r_i\}=V$. 
So for any open neighbourhood $U$  of $f(\theta)$, there exists a regular open set $V$ on $(E,\sigma (E, E^{*}))$ such that $\theta\in V\subseteq f^{-1}(U)$. Again since $(E,\sigma (E, E^{*}))$ is a locally convex space and $f$ is linear, $f\in \delta$-$(E, \sigma(E, E^{*}))^{*}$. Thus  $\delta$-$(E, \sigma (E, \delta$-$E^{*}))^{*}\subseteq \delta$-$(E, \sigma(E, E^{*}))^{*}$. Hence $\delta$-$(E, \sigma (E, \delta$-$E^{*}))^{*}= \delta$-$(E, \sigma(E, E^{*}))^{*}$.
\end{proof}

\begin{remark}\label{r1}
Let $(E, \tau)$ be a linear topological space. Then\\ $\delta$-$(E, \sigma (E, \delta$-$E^{*}))^{*}= \delta$-$(E, \sigma(E, E^{*}))^{*}=\delta$-$E^{*}$
\end{remark}

\begin{theorem}\label{th4}
Let $(E,\tau)$ and $(F, \tau_1)$ be two linear topological spaces over the field $K$ and $T: E\rightarrow F$ be linear. If $T: (E, \tau)\rightarrow (F, \tau_1)$ is $\delta$-continuous then $T: (E, \sigma(E, \delta$-$E^{*}))\rightarrow (F, \sigma(F, F^{*}))$ is $\delta$-continuous. 
\end{theorem}

\begin{proof}
Let $V$ be any basic open neighbourhood of $\theta$ on $(F, \sigma(F, F^{*}))$. Then $V=\cap_{i=1}^n\{y\in F: |g_i(y)|<r_i\}$ where $r_i>0$ and $g_i\in F^{*}$ for $i=1, 2, \cdots, n$. Now $T^{-1}(V)=T^{-1}[\cap_{i=1}^n\{y\in F: |g_i(y)|<r_i\}]=\cap_{i=1}^n T^{-1}\{y\in F: |g_i(y)|<r_i\}=\cap_{i=1}^n \{x\in E: |g_i(T(x))|<r_i\}$. To show that $T^{-1}(V)\in \sigma(E, \delta$-$E^{*})$, it is enough to show that $g_i\circ T\in \delta$-$E^{*}$ for $i\in\{1, 2, \cdots, n\}$. It is easy to see that each $g_i\circ T$ is a linear functional on $E$.
Let $i\in\{1, 2, \cdots, n\}$ and $U$ be any open set in $K$. Now  $(g_i\circ T)^{-1}(U)= T^{-1}(g_i^{-1}(U))$. Since $g_i\in F^{*}$,  $g_i^{-1}(U)$ is an open set on $(F, \tau_1)$. So, $T^{-1}(g_i^{-1}(U))$ is a delta open set on $(E, \tau)$. Thus   $g_i\circ T\in \delta$-$E^{*}$ for $i\in\{1, 2, \cdots, n\}$. This completes the proof.
\end{proof}

\begin{corollary}
Let $(E,\tau)$ and $(F, \tau_1)$ be two linear topological spaces over the field $K$ and $T: E\rightarrow F$ be linear. If $T: (E, \tau)\rightarrow (F, \tau_1)$ is $\delta$-continuous then $T: (E, \sigma(E, \delta$-$E^{*}))\rightarrow (F, \sigma(F, \delta$-$F^{*}))$ is $\delta$-continuous. 
\end{corollary}

\begin{proof}
Follows from the Theorem \ref{th4}.
\end{proof}

\begin{definition}
 A Subset $B$ of a linear topological space $(E, \tau)$ is said to be $\delta$- bounded if for every $\delta$-open neighbourhood $U$ of $\theta$ there exists $\lambda_0>0$ such that $\lambda B\subseteq U$ for all $|\lambda|\leq \lambda_0$.
\end{definition}

\begin{definition}
Let $(E,\tau)$ and $(F, \tau_1)$ be two linear topological spaces over the field $K$. A linear mapping $T: (E, \tau)\rightarrow (F, \tau_1)$ is said to be $\delta$- bounded if $T$ maps every bounded set of $E$ to a $\delta$-bounded set of $F$.
\end{definition}

\begin{theorem}
Let $(E,\tau)$  be a linear topological space over the field $K$. A set $B\subseteq E$ is $\sigma(E, \delta$-$E^{*})$ $\delta$-bounded iff for each $f\in \delta$-$E^{*}$, $\sup_{x\in B}|f(x)|<\infty$.
\end{theorem}

\begin{proof}
Let $B\subseteq E$ be a $\sigma(E, \delta$-$E^{*})$ $\delta$-bounded and $f\in \delta$-$E^{*}$. Then by the Remark \ref{r1}, $f\in\delta$-$(E, \sigma (E, \delta$-$E^{*}))^{*}$. So, $V(f,\, 1)=\{x\in E: |f(x)|<1\}$ is a  $\delta$-open neighbourhood of $\theta$ on $(E, \sigma(E, \delta$-$E^{*}))$. Since $B\subseteq E$ is $\sigma(E, \delta$-$E^{*})$ $\delta$-bounded, there exists $\lambda_0>0$ such that $\lambda_0 B\subseteq V(f,\, 1)$. So, $|f(\lambda_0\, x)|<1$ for all $x\in B$, that is, $|f( x)|<\frac{1}{\lambda_0}$ for all $x\in B$, that is, $\sup_{x\in B}|f(x)|\leq \frac{1}{\lambda_0}$.\\
Conversely, suppose the given condition is satisfied. Let $V$ be a basic $\delta$-open neighbourhood of $\theta$ in $\sigma(E, \delta$-$E^{*})$. Then there exist $f_i\in \delta$-$E^{*}$ for $i=1, 2, \cdots, n$ such that $V=\cap_{i=1}^n\{x\in E: |f_i(x)|<r_i\}$ where $r_i>0$ for $i=1, 2, \cdots, n$. Let  $\sup_{x\in B}|f_i(x)|<k_i$ for $i=1, 2, \cdots, n$. Let $\lambda_0>0$ be chosen so that $k_i\lambda_0<r_i$  for $i=1, 2, \cdots, n$. Then for $|\lambda|\leq\lambda_0$ and $x\in B$, $|f_i(\lambda x)|=|\lambda||f_i(x)|\leq\lambda_0k_i<r_i$ for all $i=1, 2, \cdots, n$. So, $\lambda x\in V$ for all $|\lambda|\leq\lambda_0$ and $x\in B$. So, $B$ is $\sigma(E, \delta$-$E^{*})$ $\delta$-bounded set.
\end{proof}

\begin{theorem}\label{th6}
Let $A$ and $B$ are non-void convex subsets of a linear topological space $E$ over $K$ and $A\cap B=\phi$. If $A$ is $\delta$-open then there exists $f\in\delta$-$E^{*}$ and $r\in \mathbb{R}$ such that $Re (f(x))<r\leq Re f(y)$ for all $x\in A$ and for all $y\in B$.
\end{theorem}

\begin{proof}
It is suffices to prove this theorem for real scalars. For, if the scalar field be $\mathbb{C}$ and $f_0$ be a $\delta$-continuous real linear functional that gives the required inequalities then $f(x)=f_0(x)-if_0(ix)$ is the required unique $\delta$-continuous complex linear functional of the theorem with $Re f=f_0$. Thus we may take $K=\mathbb{R}$.\\
Let $a_0\in A$ and $b_0\in B$ and $x_0=b_0-a_0$. Then $x_0\neq\theta$. Let $C=x_0+(A-B)$. Obviously, $\theta\in C$. Also $C$ is $\delta$-open and convex. As $C$ is open and $\theta\in C$,  the Minkowski functional $p$ of $C$ can be defined. Then $p(x+y)\leq p(x)+p(y)$ and $p(tx)=t.p(x)$ for any $x, y\in E$ and $t\geq 0$. Since $A\cap B=\phi$, $\theta\notin A- B$ and hence $x_0\notin C$. Consequently, $p(x_0)\geq 1$.\\
Let $M=\{tx_0: t\in \mathbb{R}\}$ be the one dimensional subspace spanned by $x_0$. Let $f_0:M\rightarrow K$ be a linear functional defined by $f_0(tx_0)=t$.\\
If $t\geq 0$, $f_0(t x_0)=t\leq t p(x_0)=p(t x_0)$.\\
If $t< 0$, $f_0(t x_0)=t< 0\leq p(t x_0).$ Thus $f_0(y)\leq p(y)$ for all $y\in M$. Then by Hahn Banach theorem there exists a linear functional $f$ on $E$ such that $f(x)=f_0(x)$ for $x\in M$ and $-p(-x)\leq f(x)\leq p(x)$ for all $x\in E$. Now since $C\subseteq \{x\in E: p(x)\leq 1\}$, $p(x)\leq 1$ for all $x\in C$ and hence $f(x)\leq 1$  for all $x\in C$. Again for $x\in -C$, $-x\in C$ and hence $f(-x)\leq 1$, that is $f(x)\geq -1$. Thus $|f(x)|\leq 1$ for all $x\in C\cap (-C)$. Now since $C$ is $\delta$-open, $-C$ is $\delta$-open and hence  $C\cap (-C)$ is $\delta$-open neighbourhood of $\theta$. So $f$ is bounded in a $\delta$-open neighbourhood of $\theta$.\\ 
We now show that $f$ is $\delta$-continuous. Let $B(0, \epsilon)=\{\lambda\in K: |\lambda|<\epsilon\}$ be any neighbourhood of $0$ in $K$. Let $a\in f^{-1}(B(0, \epsilon)).$ Then $|f(a)|< \epsilon$. So there exists $r> 0$ such that $|f(a)|+ r< \epsilon$.    Let $W=\frac{r}{2} V$ where $V=C\cap (-C).$ Then $W$ is also $\delta$-open neighbourhood of $\theta$ and for $\frac{r}{2}x\in W$, $|f(\frac{r}{2} x)|=\frac{r}{2} |f(x)|\leq \frac{r}{2}< r$. So, $a\in a+W\subseteq f^{-1}(B(0, \epsilon))$. Thus $f^{-1}(B(0, \epsilon))$ is a $\delta$-open neighbourhood of $a$ and hence $f$ is $\delta$-continuous.\\
Now let $a\in A$ and $b\in B$. Then $f(a)-f(b)+1=f(a)-f(b)+f(x_0)= f(a-b+x_0)\leq p(a-b+x_0)\leq 1$ as $a-b+x_0\in C$. So, $f(a)\leq f(b)$ for any $a\in A$ and $b\in B$. Also Since $f$ is linear, $f(A)$ and $f(B)$ are convex subsets of $\mathbb{R}$, that is, intervals in $\mathbb{R}$ with $f(A)$ to the left of $f(B)$. Again since any linear functional is open,  $f(A)$ is open. So there exists $r>0$ such that $f(x)<r\leq f(y)$ for $x\in A$ and $y\in B$.
\end{proof}

\begin{definition}
Let $(E, \tau)$ be a topological space and $A\subseteq E$. The delta closure of $A$ is denoted by $\delta$-$cl(A)$ and is defined by\\
$\delta$-$cl(A)=\{x\in E:\,$ every delta neighbourhood of $x$ intersects $A\}$.  
\end{definition}

\begin{note}
Let $(E, \tau)$ be a topological space and $A\subseteq E$. Then $cl(A)\subseteq \delta$-$cl(A)$.
\end{note}

\begin{theorem}
Let $(E, \tau)$ be a topological space and $A\subseteq E$. Then $\delta$-$cl(A)$ is $\delta$-closed. 
\end{theorem}

\begin{proof}
Let $x\notin \delta$-$cl(A)$. Then there exists a $\delta$-open neighbourhood $U$ of $x$ such  that $U\cap A=\phi$. Now if possible let $y\in U\cap \delta$-$cl(A)$ then $y\in U$ and $y\in \delta$-$cl(A)$ which implies that $U\cap A\neq\phi$, a contradiction. Thus $U\cap \delta$-$cl(A)=\phi$. Then $x\in U\subseteq E\setminus\delta$-$cl(A)$. Hence $E\setminus\delta$-$cl(A)$ is $\delta$-open, that is, $\delta$-$cl(A)$ is $\delta$-closed set.
\end{proof}

\begin{theorem}\label{th7}
Let $A$ and $B$ are non-void convex subsets of a locally convex space $E$ over $K$ and $A\cap B=\phi$. Then there exists $f\in\delta$-$E^{*}$ and real numbers $r_1$ and $r_2$ with $r_1< r_2$ such that $Re f(x)<r_1<r_2< Re f(y)$ for all $x\in A$ and for all $y\in B$ if and only if $\theta\notin \delta$-$cl(A- B)$.
\end{theorem}

\begin{proof}
It is suffices to prove this theorem for real scalars. For, if the scalar field be $\mathbb{C}$ and $f_0$ be a $\delta$-continuous real linear functional that gives the required inequalities then $f(x)=f_0(x)-if_0(ix)$ is the required unique $\delta$-continuous complex linear functional of the theorem with $Re f=f_0$. Thus we may take $K=\mathbb{R}$.\\
Suppose that there exists $f\in\delta$-$E^{*}$ and real numbers $r_1$ and $r_2$ with $r_1< r_2$ such that $ f(x)<r_1<r_2<  f(y)$ for all $x\in A$ and for all $y\in B$. If possible let $\theta\in \delta$-$cl(A- B)$. Then every delta neighbourhood of $\theta$ intersects $A - B$. Since $f\in\delta$-$E^{*}$, $\{x\in E: |f(x)|<r_2-r_1\}$ is a  delta neighbourhood of $\theta$. So, $\{x\in E: |f(x)|<r_2-r_1\}\cap (A-B)\neq \phi$. Therefore there exist $a\in A$ and $b\in B$ such that $|f(a-b)|<r_2-r_1$. Thus $|f(b)- f(a)|<r_2-r_1$ which contradicts our assumption.\\
Conversely, Let  $\theta\notin \delta$-$cl(A- B)$. Then $(\delta$-$cl(A- B))^c$ is an $\delta$-open neighbourhood of $\theta$. Since $E$ is locally convex, by Theorem \ref{th5} there exists an absolutely convex closed neighbourhood $U$ of $\theta$ such that $U\subseteq (\delta$-$cl(A- B))^c$. Then $int [cl \,U]\subseteq (\delta$-$cl(A- B))^c$ as $cl\, U=U$, that is, $int [cl\, U]\cap \delta$-$cl(A- B)=\phi$. Then $int [cl\, U]\cap (A- B)=\phi$. Let  $V= int [cl \,U]$. Then $V$ is regular open, that is, $V$ is $\delta$-open set and $V\cap (A-B)=\phi$ and also $V$ is convex. Then by Theorem \ref{th6}, there exists $f\in\delta$-$E^{*}$ and $r\in \mathbb{R}$ such that $f(v)<r\leq f(b)-f(a)$ for every $v\in V$, $a\in A$ and $b\in B$. Since $f$ is open and $V$ is an open convex neighbourhood of $\theta$, $f(V)$ is an open interval containing $0$ in $\mathbb{R}$. Let $\sup_{v\in V}f(v)=r_0$. Then $0<r_0\leq r$. Now $f(a)+r_0\leq f(a)+r\leq f(b)$ for all $a\in A$ and $b\in B$ implies that, $\sup_{a\in A}f(a)+r_0\leq f(b)$ for all  $b\in B$. Then  $\sup_{a\in A}f(a)+r_0\leq \inf_{b\in B}f(b)$. So for any $x\in A$ and $y\in B$ we have\\
$f(x)\leq \sup_{a\in A}f(a)< \sup_{a\in A}f(a)+\frac{1}{2}r_0=r_1< \sup_{a\in A}f(a)+\frac{2}{3}r_0=r_2< \sup_{a\in A}f(a)+r_0\leq\inf_{b\in B}f(b)\leq f(y)$, that is, $f(x)<r_1<r_2< f(y)$ for every $x\in A$ and $y\in B$. 
\end{proof}

\begin{theorem}
Let $\tau_1$ and $\tau_2$ be two locally convex topological spaces on a linear space $E$ over $K$ such that they have the same $\delta$-continuous linear functionals, that is, $\delta$-$(E, \tau_1)^{*}=\delta$-$(E, \tau_2)^{*}$. Then an absolutely convex set $A\subseteq E$ is $\tau_1$ $\delta$-closed = $\tau_2$ $\delta$-closed. 
\end{theorem}

\begin{proof}
Suppose a convex set $A\subseteq E$ is $\tau_1$ $\delta$-closed and $a\notin A$. Then $A-a$ is $\tau_1$ $\delta$-closed and $\theta\notin A-a$, that is, $\theta\notin \delta$-$cl_{\tau_1}(A-a)$. Then by Theorem  \ref{th7}, there exists $\tau_1$ $\delta$-continuous functional $f$ and real numbers $r_1$, $r_2$ with $r_1<r_2$ such that $Re f(x)<r_1<r_2< Re f(a)$ for all $x\in A$. So, $\sup_{x\in A}Re f(x)\leq r_1<r_2< Re f(a)$. Since $A$ is  absolutely convex  and $f$ is linear, $f(A)$ an is  absolutely convex subset of $K$. So, $f(A)$ is a circular disc with centre at $0$. Thus $|f(x)|\leq r_1<r_2< Re f(a)\leq |f(a)|$ for all $x\in A$, that is, $\sup_{x\in A}|f(x)|\leq r_1<r_2< |f(a)|$. So there exists $\epsilon>0$ such that $\sup_{x\in A}|f(x)|+ \epsilon < |f(a)|$. So, $|f(x)|+ \epsilon < |f(a)|$ for all $x\in A$, or, $|f(x)-f(a)|\geq |f(a)|-|f(x)|>\epsilon$ for all $x\in A$. Let $U=\{z\in E:\, |f(z)|<\epsilon\}$. By hypothesis $f$ is a $\tau_2$ $\delta$-continuous functional and hence $U$ is a $\tau_2$ $\delta$-open neighbourhood of $\theta$. So $a+U$ is a $\tau_2$ $\delta$-open neighbourhood of $a$. We now show that $(a+U)\cap A=\phi$. If $b\in A$ with $a+u=b$ for some $u\in U$. Then $|f(b)-f(a)|>\epsilon$, that is, $|f(b-a)|>\epsilon$, that is, $|f(u)|>\epsilon$ for some $u\in U$ which is a contradiction. Thus $a+U$ is a $\tau_2$ $\delta$-open neighbourhood of $a$ disjoint from $A$. Hence $A$ is $\tau_2$ $\delta$-closed.\\
 Interchanging $1$ and $2$, we can similarly show that $A$ is $\tau_2$ $\delta$-closed implies $A$ is $\tau_1$ $\delta$-closed.
\end{proof}

\begin{theorem}\label{th8}
Let $(E,\tau)$  be a locally convex space over the field $K$. If $B$ is a convex subset of $E$, then $\sigma(E, \delta$-$E^{*})$-closure of $B=$delta $\tau$-closure of $B$, that is, $cl_{\sigma(E, \delta-E^{*})}(B)=\delta$-$cl_\tau(B)$.
\end{theorem}

\begin{proof}
 Let $a\notin \delta$-$cl_\tau(B)$. Since $B$ and $\{a\}$ are disjoint convex sets in $E$ and $\theta\notin \delta$-$cl_\tau(B)-a=\delta$-$cl_\tau( \delta$-$cl_\tau(B)-a)$, by Theorem \ref{th7}, there exist $f\in\delta$-$E^{*}$ and real numbers $r_1$ and $r_2$ with $r_1< r_2$ such that $Re f(x)<r_1<r_2< Re f(a)$ for all $x\in \delta$-$cl_\tau(B)$. Since  $f\in\delta$-$E^{*}$,  by Remark \ref{r1}, $f\in \delta$-$(E, \sigma (E, \delta$-$E^{*}))^{*}$. Then $Re f\in \delta$-$(E, \sigma (E, \delta$-$E^{*}))^{*}$. So $U=\{x\in E:\,|Re f(x)|<r_2-r_1\}$ is a delta $\sigma (E, \delta$-$E^{*})$ neighbourhood of $\theta$. Then $a+U$ is a delta $\sigma (E, \delta$-$E^{*})$ neighbourhood of $a$. We now show that $(a+U)\cap B=\phi$. If possible let $b=a+u$ for some $b\in B$ and $u\in U$. Then $|Re f(b-a)|<r_2-r_1$, that is, $|Re f(b)-Re f(a)|<r_2-r_1$ for some $b\in B\subseteq \delta$-$cl_\tau(B)$ which is a contradiction. Thus $a\notin cl_{\sigma(E, \delta-E^{*})}(B)$. So, $cl_{\sigma(E, \delta-E^{*})}(B)\subseteq \delta$-$cl_\tau(B)$.\\
Again let $a\in\delta$-$cl_\tau(B)$ and $U$ be any $\sigma(E, \delta$-$E^{*})$ open neighbourhood of  $\theta$. Then there exist $f_i\in \delta$-$E^{*}$ and $r_i>0$ for $i=1, 2, \cdots, n$ such that $U=\cap_{i=1}^n\{x\in E: f_i(x)<r_i\}$. Since $f_i\in \delta$-$E^{*}, \,\{x\in E: f_i(x)<r_i\}$ is delta open set on $(E, \tau)$ for $i=1, 2, \cdots, n$. So, $U$ is delta open set on $(E, \tau)$ containing $\theta$ and hence $(a+U)\cap B\neq\phi$. Therefore $a\in cl_{\sigma(E, \delta-E^{*})}(B)$. Thus $\delta$-$cl_\tau(B)\subseteq cl_{\sigma(E, \delta-E^{*})}(B)$. Hence $cl_{\sigma(E, \delta-E^{*})}(B)=\delta$-$cl_\tau(B)$.
\end{proof}

\begin{corollary}
If $(E, \tau)$ be a locally convex space then\\
$(i)$ a convex subset $B$ of $E$ is delta $\tau$-closed if and only if it is $\sigma(E, \delta$-$E^{*})$-closed.
\end{corollary}

\begin{proof}
Follows from the Theorem \ref{th8}.
\end{proof}

\section*{Conclusion}
To define delta weak topology, we have used the definition of $\delta$-continuity as stated by  R. M. Latif \cite{Latif}. Then find out the characteristic of the set of $\delta$-continuities and $\delta$-closure of a set on delta weak topological spaces.    
So, one can try to define another type of $\Gamma$-topology with the help of T. Noiri's \cite{Noiri} definition of $\delta$-continuity and to find out the behaviour of various properties on these spaces.

\section*{Acknowledgement}
I acknowledge my wife Mrs. Krishna Roy for her valuable suggestions in linguistic and grammatical parts of this paper.

\end{document}